\documentclass[12pt]{amsart}
\usepackage{amsfonts,amssymb,amsmath,amsthm,mathrsfs}

\usepackage{verbatim}
\usepackage{hyperref}
\usepackage{color}

\begin{document}

\newtheorem{theorem}{Theorem}    
\newtheorem{proposition}[theorem]{Proposition}
\newtheorem{conjecture}[theorem]{Conjecture}
\def\theconjecture{\unskip}
\newtheorem{corollary}[theorem]{Corollary}
\newtheorem{lemma}[theorem]{Lemma}
\newtheorem{sublemma}[theorem]{Sublemma}
\newtheorem{observation}[theorem]{Observation}
\theoremstyle{definition}
\newtheorem{definition}{Definition}
\newtheorem{notation}[definition]{Notation}
\newtheorem{remark}[definition]{Remark}
\newtheorem{question}[definition]{Question}
\newtheorem{questions}[definition]{Questions}
\newtheorem{example}[definition]{Example}
\newtheorem{problem}[definition]{Problem}
\newtheorem{exercise}[definition]{Exercise}
\newtheorem*{tha}{Theorem A}
\newtheorem*{thb}{Theorem B}
\numberwithin{theorem}{section}
\numberwithin{definition}{section}
\numberwithin{equation}{section}
\def\cic{\mathbf}

\def\earrow{{\mathbf e}}
\def\rarrow{{\mathbf r}}
\def\uarrow{{\mathbf u}}
\def\varrow{{\mathbf V}}
\def\tpar{T_{\rm par}}
\def\apar{A_{\rm par}}

\def\reals{{\mathbb R}}
\def\torus{{\mathbb T}}
\def\heis{{\mathbb H}}
\def\integers{{\mathbb Z}}
\def\naturals{{\mathbb N}}
\def\complex{{\mathbb C}\/}
\def\distance{\operatorname{distance}\,}
\def\support{\operatorname{support}\,}
\def\dist{\operatorname{dist}\,}
\def\Span{\operatorname{span}\,}
\def\degree{\operatorname{degree}\,}
\def\kernel{\operatorname{kernel}\,}
\def\dim{\operatorname{dim}\,}
\def\codim{\operatorname{codim}}
\def\trace{\operatorname{trace\,}}
\def\Span{\operatorname{span}\,}
\def\dimension{\operatorname{dimension}\,}
\def\codimension{\operatorname{codimension}\,}
\def\nullspace{\scriptk}
\def\kernel{\operatorname{Ker}}
\def\ZZ{ {\mathbb Z} }
\def\p{\partial}
\def\rp{{ ^{-1} }}
\def\Re{\operatorname{Re\,} }
\def\Im{\operatorname{Im\,} }
\def\ov{\overline}
\def\eps{\varepsilon}
\def\lt{L^2}
\def\diver{\operatorname{div}}
\def\curl{\operatorname{curl}}
\def\etta{\eta}
\newcommand{\norm}[1]{ \|  #1 \|}
\def\expect{\mathbb E}
\def\bull{$\bullet$\ }

\def\xone{x_1}
\def\xtwo{x_2}
\def\xq{x_2+x_1^2}
\newcommand{\abr}[1]{ \langle  #1 \rangle}
\def\Z {{\mathbb Z}}
\def \l {\langle}
\def \r {\rangle}
\def \pt {\partial_t}
\def \ptt {\partial_{tt}}
\def \ps {\partial_s}
\def \and{\qquad\text{and}\qquad}
\def\mb{{ \scriptscriptstyle\bullet}}
\def\st {\,:\,}

\newcommand{\Norm}[1]{ \left\|  #1 \right\| }
\newcommand{\set}[1]{ \left\{ #1 \right\} }
\def\one{\mathbf 1}
\def\whole{\mathbf V}
\newcommand{\modulo}[2]{[#1]_{#2}}

\def\scriptf{{\mathcal F}}
\def\scriptg{{\mathcal G}}
\def\scriptm{{\mathcal M}}
\def\scriptb{{\mathcal B}}
\def\scriptc{{\mathcal C}}
\def\scriptt{{\mathcal T}}
\def\scripti{{\mathcal I}}
\def\scripte{{\mathcal E}}
\def\scriptv{{\mathcal V}}
\def\scriptw{{\mathcal W}}
\def\scriptu{{\mathcal U}}
\def\scriptS{{\mathcal S}}
\def\scripta{{\mathcal A}}
\def\scriptr{{\mathcal R}}
\def\scripto{{\mathcal O}}
\def\scripth{{\mathcal H}}
\def\scriptd{{\mathcal D}}
\def\scriptl{{\mathcal L}}
\def\scriptn{{\mathcal N}}
\def\scriptp{{\mathcal P}}
\def\scriptk{{\mathcal K}}
\def\frakv{{\mathfrak V}}
\def\R {\mathbb{R}}

\author{Loukas Grafakos}
\address{
	Loukas Grafakos
	\\
	Department of Mathematics
	\\
   University of Missouri
	\\
	Columbia MO 65211
	\\
USA
}
\email{grafakosl@missouri.edu}

\thanks {{\it{Mathematics Subject Classification.}} Primary 42B20; Secondary 42B25\\
\indent	{\it Key words and phrases}. Sparse domination, rough bilinear singular integrals,
$A_{\vec{p},\vec{r}}$ weights.\\
\indent The first author was supported by the Simons Foundation under award number 624733.  
The third named author was supported partly by NSFC (Nos. 11671039, 11871101) and NSFC-DFG (No. 11761131002).\\
\indent
Corresponding author: Qingying Xue}

\author{Zhidan Wang}
\address{
	Zhidan Wang
	\\
	School of Mathematical Sciences
	\\
	Beijing Normal University
	\\
	Laboratory of Mathematics and Complex Systems
	\\
	Ministry of Education
	\\
	Beijing 100875
	\\
	People's Republic of China
}
\email{zdwang@mail.bnu.edu.cn}
\author{Qingying Xue}
\address{Qingying Xue
	\\
	School of Mathematical Sciences
	\\
	Beijing Normal University
	\\
	Laboratory of Mathematics and Complex Systems
	\\
	Ministry of Education
	\\
	Beijing 100875
	\\
	People's Republic of China
}
\email{qyxue@bnu.edu.cn}


\date{\today}

\title
[ sparse domination ]
{sparse domination and weighted estimates for rough  bilinear  singular integrals  }

\begin{abstract}
Let $r>\frac{4}{3}$ and let $\Omega \in L^{r}(\mathbb{S}^{2n-1})$ have vanishing integral. We show that the   bilinear rough singular integral
$$T_{\Omega}(f,g)(x)=
\textrm{p.v.} \int_{\mathbb{R}^n}\int_{\mathbb{R}^n}\frac{\Omega((y,z)/|(y,z)|)}{|(y,z)|^{2n}}f(x-y)g(x-z)\,dydz,$$ satisfies a sparse bound by $(p,p,p)$-averages, where $p$ is bigger than a certain number explicitly related to $r$ and $n$. As a consequence we deduce
certain quantitative weighted estimates
for bilinear homogeneous singular integrals associated with rough homogeneous kernels.
\end{abstract}

\maketitle

\section{Introduction}
In   1952, Calder\'{o}n and Zygmund \cite{calderon1952} established the
 existence and $L^p(\mathbb{R}^n)$ boundedness   of the following rough singular integrals 
$$
T_K(f)(x_1,x_2,\dots,x_n)=\int_{\mathbb{R}^n} f (s_1, \dots , s_n)K(x_1-s_1,\dots,x_n-s_n)ds_1 \cdots ds_n,
$$ 
where $f$ is an integrable function defined on $\mathbb {R}^n$ and
$$
K(x_1,   \dots,  x_n) = \rho ^{-n}\Omega (\alpha _1,  \dots,\alpha _n) ,
$$
with $x_j=\rho \cos \alpha _j $ for all $j$, $\rho>0$, 
and $\alpha _1 ,\alpha _2,\dots,\alpha _n$ are the direction angles of $(x_1, x_2, \dots,  x_n)$.
Later on, using the method of rotations, Calder\'{o}n and Zygmund  \cite{10.2307/2372517}   proved that the operator  
$$
T_{\Omega}(f)(x)=\textrm{p.v.}\int_{\mathbb{R}^n}\frac{\Omega(y/|y|)}{|y|^n}f(x-y)dy
$$
is bounded on $L^p(\mathbb{R}^n)$ $(1<p<\infty)$ whenever $\Omega\in L^1{(\mathbb{S}^{n-1})}$,  $\int_{\mathbb{S}^{n-1}}\Omega \,d\sigma=0 $ and  if the even part of $\Omega$ belongs to the class $ L\log L(\mathbb{S}^{n-1})$.

\medskip
Since 1956 this area has flourished and has been enriched by activity that is too big to list here. We note however the work of Christ \cite{Ch88},  Christ and Rubio~de Francia \cite{CRub}, 
Seeger  \cite {Seeger1996}, Tao \cite {{10.2307/24901023}},
Duoandikoetxea and Rubio~de Francia \cite{DR}, Grafakos and Stefanov \cite {GS2} among 
many others. 
The weighted theory of $T_\Omega$ is also quite rich; here we note the   work 
of Duoandikoetxea \cite{Duo1993} and Vargas \cite{V}  and we would like to direct  attention   to the recent works of 
\cite{Ding2019,Li2019,Liu2019}.

In order to state more known results, we first introduce some notation.
A collection ${{\mathcal S}}$ of cubes in $\mathbb{R}^n$  is called $\eta$-sparse if for each $Q\in {{\mathcal S}}$ there is $E_Q\subset Q$ such that $|E_Q|\geq \eta |Q|$, and such that $E_Q\cap E_{Q'}=\varnothing$ when $Q\neq Q'$ (here $0<\eta<1$). For an    $\eta$-sparse collection of cubes ${{\mathcal S}}$  we  use the notation
$$
\mathsf{PSF}_{\mathcal S;p_1,p_2}(f_1,f_2) := \sum_{Q\in \mathcal S} |Q| \l f_1 \r_{p_1,Q}   \l f_2 \r_{p_2,Q}, \quad \l f \r_{p,Q}:=  |Q|^{-\frac1p}\left\|f\cic{1}_{Q}\right\|_{L^p}.
$$
 Such expressions dominate quantities $|\langle T(f_1), f_2\rangle|$ for linear operators $T$. 
 This type of domination is called  sparse   and plays an important role
and finds  wide  applicability  in harmonic analysis. 
For instance, it was used in the proof of $A_2$ conjecture \cite{La,Ler2013}. 
Earlier   works  related to sparse domination can be found in \cite{Barron2019,HRT, La, Ler, li2018,Volberg_2018} and the references therein. In 2017, Conde-Alonso et al. \cite{Conde-Alonso2017} obtained the following sparse domination for $T_{\Omega}$:
\begin{eqnarray*}|\langle T_{\Omega}(f_1),f_2\rangle| \leq \frac{Cp}{p-1} \sup_{{{\mathcal S}}}  \mathsf{PSF}_{{{\mathcal S}};1,p}(f_1,f_2)
	\begin{cases}
		\|\Omega\|_{L^{q,1}\log L(\mathbb S^{d-1})} , &1<q<\infty, p\geq q';\cr \|\Omega\|_{L^{\infty}(\mathbb  S^{d-1})}, &1<p<\infty .\end{cases}
\end{eqnarray*}
As a consequence, the authors in \cite{Conde-Alonso2017} deduced a  new sharp quantitative $A_p$-weighted estimate for $ T_{\Omega}$. Subsequently, for all $\epsilon >0$, Di Plinio, Hyt\"{o}nen, and Li \cite{plinio2017sparse},  provided a sparse bound by $(1+\epsilon,1+\epsilon)$-averages with linear growth in $\epsilon ^{-1}$ for the associated maximal truncated singular integrals $T_{*} $, i.e., 
$\|T_{*}\|_{(1+\epsilon,1+\epsilon),sparse} \leq {C}{\epsilon ^{-1}}.$
As  a corollary, certain novel quantitative weighted norm estimates were given for $T_{*}$.
\medskip

The study of bilinear singular integrals originated in the celebrated work of Coifman and Meyer \cite{Coifman1975}. The main object of study  is the bilinear operator (which is denoted as in the linear case without risk of confusion as its linear counterpart will not appear in the sequel)
\begin{align}\label{T}
T_{\Omega}(f,g)(x)=
\textrm{p.v.} \int_{\mathbb{R}^n}\int_{\mathbb{R}^n}\frac {\Omega((y,z)/|(y,z)|)}{|(y,z)|^{2n}}f(x-y)g(x-z)\,dydz,
\end{align}
where $\Omega$ is an integrable function on $\mathbb S^{2n-1}$ with mean value zero. 
Let $1<p_1,p_2<\infty$ and $\frac{1}{p}=\frac{1}{p_1}+\frac{1}{p_2}$. In 2015, Grafakos, He and Honz\'{\i}k \cite{Grafakos2018} obtained the $L^{p_1}({\mathbb{R}^n})\times L^{p_2}({\mathbb{R}^n})$ to $L^{p}({\mathbb{R}^n})$ boundedness for $T_{\Omega}$ when $\Omega \in L^{\infty}(\mathbb{S}^{2n-1})$. Additionally, these authors showed that $T_{\Omega}$ is bounded from $L^2 ({\mathbb{R}^n})\times L^2({\mathbb{R}^n})$  to $ L^1({\mathbb{R}^n})$ if $\Omega \in L^{q}(\mathbb{S}^{2n-1})$ for $q\geq 2$.  In 2018, Grafakos, He, and 
Slav\'ikov\'a \cite{Grafakos2019} gave a criterion  for $L^2({\mathbb{R}^n})\times L^2({\mathbb{R}^n})$ to $L^1({\mathbb{R}^n})$ boundedness for certain bilinear operators. As an application, these authors improved the results in \cite{Grafakos2018} as follows:
\begin{tha}\label{ta} (\cite{Grafakos2019})
	Let $q>4/3$ and  $ \Omega \in L^{q} (\mathbb{S}^{2n-1})$ with $\int_{\mathbb{S}^{2n-1}}\Omega\, 
	 d\sigma =0$. Then  
	$	\|T_{ \Omega}\|_{L^{p_1}(\mathbb{R}^n) \times L^{p_2}(\mathbb{R}^n) \rightarrow  L^p(\mathbb{R} ^n) }< {\infty}$
	whenever $2 \leq p_1, p_2 \leq  \infty $, $1\le p\le 2$,  and $\frac{1}{p} = \frac{1}{p_1} +\frac{1}{p_2}$.
\end{tha}

For   $\Omega$ in $L^q(\mathbb{S}^{2n-1})$, it is natural to ask  for the exact 
range of $(p_1,p_2,p)$ such that $T_{ \Omega}$ maps $L^{p_1}(\mathbb{R}^n) \times L^{p_2}(\mathbb{R}^n)$ to $ L^{p}(\mathbb{R} ^n) $. This problem is quite delicate. 
A counterexample of Grafakos, He and Slav\'{\i}kov\'{a} \cite{Grafakos2019a} shows that there exists an $\Omega $ in $L^{q} (\mathbb{S}^{2n-1})$, $1\leq q<\infty$, which satisfies the H\"{o}rmander kernel condition on $\mathbb R^{2n}$, such that the associated $T_\Omega$ is unbounded from 
$L^{p_1}(\mathbb{R}^n) \times L^{p_2}(\mathbb{R}^n)$ to $ L^{p }(\mathbb{R} ^n) $ when 
$\frac{1}{p} = \frac{1}{p_1} +\frac{1}{p_2}$, $1 \leq p_1, p_2 \leq  \infty $ and $\frac{1}{p}+\frac{2n-1}{q}>2n$. However, it is   unknown whether $T_\Omega$ is bounded  when the last condition fails. 
\smallskip

In this work, we focus on the sparse domination of  $T_{ \Omega}$ for rough functions $\Omega$. 
Note that the authors in \cite{Culiuc2018} established a uniform domination of the family of trilinear multiplier forms with singularity over an one-dimensional subspace. Later Barron \cite{Barron2017} considered the sparse domination for rough bilinear singular integrals with $\Omega$ in $L^\infty(\mathbb S^{2n-1})$.
\begin{thb}\label{tb} (\cite{Barron2017})
	Suppose $T_{\Omega}$ is the rough bilinear singular integral operator  defined by \eqref{T}, with $\Omega \in L^{\infty}(\mathbb{S}^{2n-1})$ and $\int_{\mathbb{S}^{2n-1}}\Omega \, d\sigma=0$. Then for any $1<
	p<\infty$, there is a constant $C_{p,n} >0$ so that
	$$|\langle T_{\Omega}(f_1, f_2), f_3 \rangle| \leq C_{p,n}\|\Omega\|_{L^{\infty}(\mathbb S^{2n-1})} \sup _{{\mathcal S}}\mathsf{PSF}_{{\mathcal S}}^{(p,p,p)}(f_1, f_2, f_3),$$ where the sparse $(p_1,p_2,p_3)$-averaging form is defined as
	$$
	\mathsf{PSF}^{(p_1,p_2,p_3)}_{{\mathcal S}} (f_1,f_2,f_3):=\sum_{Q\in {{{\mathcal S}}}}|Q|\prod\limits_{i=1}^3\langle f_i \rangle _{p_i,Q}, \hbox{\ for \ }
1\le p_i<\infty,  \ i=1,2,3.	$$
\end{thb}

\medskip

In this paper, we  establish sparse domination for bilinear rough operator $T_{\Omega}$ with $\Omega \in L^{r}(\mathbb S^{2n-1})$ for $r<\infty$. These  $\Omega$ produce rougher singular integrals than the ones  previously studied. As a result we deduce certain quantitative weighted estimates for rough bilinear singular integral operators.  
The main result  of this paper is as follows:

\begin{theorem}\label{thm1} Let $\Omega \in L^{r}(\mathbb{S}^{2n-1})$, $r>4/3$, and $\int _{\mathbb{S}^{2n-1}}\Omega=0$. Let
	$T_{\Omega}$ be the rough bilinear singular integral operator defined in \eqref{T}. Then for  $p>\max \{\frac{24n+3r-4}{8n+3r-4} , \frac{24n+r}{8n+r} \}$ 
 there exists a constant $C=C_{p,n,r}$ such that 
$$
|\langle T_{\Omega}(f_1,f_2),f_3 \rangle|\leq C\|\Omega\|_{L^{r}(\mathbb{S}^{2n-1})}\sup _{{\mathcal S}} \mathsf{PSF}_{{\mathcal S}}^{(p,p,p)}(f_1,f_2,f_3).
$$
\end{theorem}
\begin{remark} Letting $r\to\infty$,  the restriction on $p$ in Theorem~\ref{thm1} becomes
$p>1$ for $\Omega\in L^\infty(\mathbb{S}^{2n-1})$. Thus Theorem~\ref{thm1} coincides with the sparse domination result of Theorem B when $r=\infty$.  Thus our work essentially extends that of \cite{Barron2017} and all the weighted results it implies. 
\end{remark}
In order to state our corollaries, we  recall some background and  introduce 
notation relevant  to certain classes of weights.
Let $p'=p/(p-1)$ be the dual exponent of $p$. We recall the definition of the $A_p$ weight classes: We say $w\in A_p$ for $1<p<\infty$ if $w>0$, $w\in L_{loc}^1$ and $$[w]_{A_p}:=\sup\limits_{Q}\Big(\frac{1}{|Q|}\int_{Q}w\Big)\,\Big(\frac{1}{|Q|}\int_{Q}{w}^{- \frac{1}{p-1}}\Big)^{p-1}<\infty.$$

In 2002  Grafakos and Torres \cite{{GraTor2}} initiated the weighted theory for the multilinear singular operators but it was not until   2009 that  Lerner et. al. \cite{Lerner2009}  introduced the canonical Muckenhoupt weight class $A_{\vec{p}}$, which provides a natural 
analogue of the linear theory.
\medskip

\begin{definition}[{\bf Multiple weight class $A_{\vec{p}}$}, \cite{Lerner2009}]Let $1\leq p_1,\dots,p_m<\infty,\ \vec{w}=(w_1,\dots,w_m),$\ where $w_i\ (i=1,\dots,m)$ are nonnegative functions defined on $\mathbb R^n$,\ and denote $v_{\vec{w}}=\mathop\prod\limits_{j=1}^mw_j^{{p}/{p_j}}.$\ We say $\vec{w}\in A_{\vec{p}}$  if
$$
[{\vec{w}}]_{A_{\vec{p}}}=\sup_Q\left(\frac{1}{|Q|}\int_Qv_{\vec{w}}(t)dt\right)^{\frac{1}{p}}\prod_{i=1}^m\left(\frac{1}{|Q|}\int_Q w_i^{1-p_i'}(t)dt\right)^{\frac{1}{p_i'}}<\infty,
$$
where the supremum is taken over all cubes $Q\subset\mathbb R^n$,\
and the term $\left(\frac{1}{|Q|}\int_Qw_i^{1-p_i'}(t)dt\right)^{\frac{1}{p_i'}}$ is understood as $(\inf_Qw_i)^{-1}  $ when $p_i=1$.
\end{definition}
More general weights class than $A_{\vec{p}}$ has also been considered by Li, Martell, and Ombrosi in \cite{LMO2018}. 
For $m\geq 2$,
given $\vec{p}=(p_1,\dots,p_m)$ with $1\leq p_1,\dots,p_m< \infty$ and $\vec{r}=(r_1,\dots,r_{m+1})$ with $1\leq r_1,\dots,r_{m+1 }<\infty$, we say that $\vec{r} \prec \vec{p}$
whenever 
$$
\textup{$r_i < p_i$, $i=1,\dots,m$ and $r'_{m+1}>p$, where $\frac{1}{p}:=\frac{1}{p_1}+\dots+\frac{1}{p_m}$.}
$$

\begin{definition}[{\bf $A_{\vec{p},\vec{r}}$ weight class}, \cite{LMO2018}] \label {def1}Let $m\geq 2$ be an integer, $\vec{p}=(p_1,\dots , p_m)$
	with $1\leq p_1,\dots,p_m <\infty$ and $\vec{r}=(r_1,\,\dots,\,r_{m+1})$ with $1\leq r_1,\dots,r_{m+1} <\infty$.
	$1/p=\sum_{k=1}^{m}1/p_k$. For each $w_k>0$, $w_k\in L_{loc}^1$,
	set
	$$
	{{w}}=\prod_{k=1}^{m}w_k^{p/p_k} .
	$$
	 We say that $\vec{w}=(w_1,...,w_m)\in
	A_{\vec{p},\vec{r}}$ if $0<w_i<\infty$, $1 \le i\le m$ and $[w]_{A_{\vec{p},\vec{r}}}<\infty$
	with
	$$
	[\vec w]_{A_{\vec{p},\vec{r}}}=\sup_{Q}\Big(\frac{1}{|Q|}\int_Q {{w(x)}^{\frac{r'_{m+1}}{r'_{m+1}-p}}}\,{\rm d}x\Big)^{1/p-1/r'_{m+1}}
	\prod_{k=1}^m\Big(\frac{1}{|Q|}\int_Qw_k(x)^{-\frac{1}{\frac{p_k}{r_k}-1}}
	\,{\rm d}x\Big)^{1/r_k-1/p_k}.
	$$
	
	When $r_{m+1}=1$ the term corresponding to $w$ needs to be replaced by $(\frac{1}{|Q|}\int_Q w dx)^{\frac{1}{p}}$.
Here and afterwards,  the expression
$$
\Big(\frac{1}{|Q|}\int_Q w_k(x)^{-\frac{1}{\frac{p_k}{r_k}-1}}\,{\rm d}x\Big)^{1/r_k-1/p_k}
$$  
 is understood as ess$\sup_{ Q}w_k^{-1/p_k}  $ when $p_k=r_k$.
 
		 When $r_1=\dots=r_m=1$, $A_{\vec{p},\vec{r}}$ coincides with
 the weight class $A_{\vec{p}}$ introduced by
	Lerner et al. \cite{Lerner2009}
\end{definition}
As an application of the sparse domination, we obtain some weighted estimates for $T_{\Omega}$. {The first result concerns with the multiple weights and the other one is associated with one weight case.}
\begin{corollary}\label{collary2}
Let $\Omega \in L^{r}(\mathbb{S}^{2n-1})$ with $r>4/3$ and  $\int _{\mathbb{S}^{2n-1}}\Omega \, d\sigma=0$.  Let $\vec{q}=(q_1,q_2)$, $\vec{p}=(p_1,p_2,p_3)$ with $\vec{p}\prec \vec{q}$ and $p_i>\max\{\frac{24n+3r-4}{8n+3r-4}, 	\frac{24n+r}{8n+r} \}$,  $i=1,2,3$. Let 
$$
\mu_{\vec{v}}=\prod_{k=1}^{2}v_k^{q/q_k}
$$
 and $\frac{1}{q}=\frac{1}{q_1}+\frac{1}{q_2}$, $1< q <\max\{\frac{24n+3r-4}{16n}, \frac{24n+r}{16n}\}$ and 
let $q_3=q'$. Then
there is a constant $C=C_{\vec p ,\vec q, r,n}$ such that
	$$
	\|T_{\Omega}(f,g)\|_{L^{q}(\mu_{\vec{v}})}\leq C \|\Omega\|_{L^r} [\vec{v}]^{\max_{1 \leq i\leq 3}\{\frac{p_i}{q_i-p_i}\}}_{A_{\vec q,\vec p}}\|f\|_{L^{q_1}(v_1)}\|g\|_{L^{q_2}(v_2)}.
	$$
\end{corollary}
\begin{corollary}\label{collary1}
	Let $\Omega \in L^{r}(\mathbb{S}^{2n-1})$ with $r>4/3$ and  $\int _{\mathbb{S}^{2n-1}}\Omega\, d\sigma=0$. For $w \in A_{p/2}$, $\max\{2,\frac{24n+3r-4}{8n+3r-4}, 	\frac{24n+r}{8n+r} \}<p<\max\{\frac{24n+3r-4}{8n},\frac{24n+r}{8n} \}$,  there exists a constant $C=C_{w,p,n,r }$ such that
	$$\|T_{\Omega}(f_1,f_2)\|_{L^{p/2}(w)}\leq C \|\Omega\|_{L^r}\|f_1\|_{L^p(w)}\|f_2\|_{L^p(w)}.$$
\end{corollary}

\begin{remark}We make few comments about  Corollaries \ref{collary2} and \ref{collary1}. 
\begin{itemize}
	\item The class of weights   in Corollary \ref{collary2} is slightly different 
	than that  used in \cite{Barron2017}.
 	\item In Theorem A there is a restriction  $p_i>2$. It is interesting that in Corollary~\ref{collary2}, when $\frac{4}{3}<r< 8n$ 
it is easy to see that	$p_i>2$, $i=1,2$. However, when $r\geq 8n$, then $p_1$, $p_2$ could be   smaller than $2$. This means that, in some sense, $q_i$ enjoys more freedom in 
Corollary~\ref{collary2}, since we only require $q>1$ and there is no need to assume that each $q_i>2.$
	\item We guess that the index regions in the above two corollaries are far from optimal. To find the best region for the above weighted results should be a very interesting problem.
\end{itemize}
\end{remark}	

The article is organized as follows.  Section \ref{sect2} contains   definitions and basic lemmas. 
An analysis of  the Calder\'{o}n-Zygmund kernel is given  in Section \ref{sec3}.
Section \ref{sec4}  and  Section \ref{sec5} are devoted to the demonstration of the proof of Theorem \ref{thm1} 
and its corollaries.    Throughout this paper, the notation $\lesssim$ will be used to denote an inequality with  an
inessential constant on the right.  We denote by $\ell (Q)$ the side length of a   cube $Q $ in $ \mathbb{R}^n$ and by diam$(Q)$ its diameter. For $\lambda>0$  we use the notation $\lambda Q $ for  the cube with the same center as $ Q$ and   side length $\lambda\ell(Q)$.

\section{ Definitions and main lemmas} \label{sect2}
In this section we   consider a   general bilinear operator that commutes with translations
\begin{equation}\label{BOCT}
T[K](f_1,f_2)(x) = \textup{p.v.} \int_{\mathbb R^n}  \int_{\mathbb R^n} 
K(x-x_1,x-x_2) f_1(x_1) f_2(x_2) \, dx_1\, dx_2
\end{equation} 
 and assume it is a bounded  bilinear operator mapping $L^{r_1}(\mathbb{R}^n) \times L^{r_2}(\mathbb{R}^n)\rightarrow L^{\alpha}(\mathbb{R}^n)$ for some $r_1,r_2,  \alpha \geq 1$ with $\frac{1}{r_1}+\frac{1}{r_2}=\frac{1}{\alpha}$.  
It is assumed that the kernel $K$ of $T[K]$  has a decomposition of the form
\begin{align}\label{K1}
K(u,v )=\sum\limits_{s\in \mathbb{Z}}K_s(u,v ),
\end{align}
where $K_s$ {is a smooth truncation of $K$ that} enjoys the property
\begin{align*}
\textup{ supp} K_s \subset \big\{(u,v)\in \mathbb{R}^{2n} :\,\, 2^{s-2}<|u|<2^s,\, 2^{s-2}<|v|<2^s   \big\}.
\end{align*}

 The truncation of $T[K]$ is defined as
\begin{align}
T[K]_{t_1}^{t_2} (f_1,f_2)(x ):=\sum\limits_{t_1<s<t_2}\int_{\mathbb{R}^{ n}}\int_{\mathbb{R}^{ n}}K_s(x-x_1,x-x_2)f_1(x_1)f_2(x_2)\, dx_1dx_2,
\end{align}
where $0<t_1<t_2<\infty$.   
See Section 2.1 in \cite{Barron2017} for remarks on this type of truncated operators.
In this work, we assume that the truncated norm satisfies
\begin{align}\label{1}
\sup_{0<t_1<t_2<\infty}  \|T[K]_{t_1}^{t_2} \|_{{L^{r_1}\times L^{r_2} }\rightarrow L^{\alpha}}<\infty,
\end{align}
 for some $r_1,r_2,\alpha\ge 1$ satisfying $\frac{1}{r_1}+\frac{1}{r_2}=\frac{1}{\alpha}$. 
To study  bilinear operators $T$,  we  often work   with the trilinear form of the type $\langle T(f_1,f_2),f_3\rangle=\int_{\mathbb{R}^n}T(f_1,f_2)f_3(x)\,dx$. In our case, the trilinear truncated form is 
$$
\langle T[K]_{t_1}^{t_2}(f_1,f_2),f_3 \rangle=\int_{\mathbb R^n}  T[K]_{t_1}^{t_2}(f_1,f_2) f_3 \, dx .
$$
Denoting  by $C_T(r_1,r_2,\alpha)$
 the following constant
\begin{equation}\label{CT33}
C_T(r_1,r_2,\alpha):=\sup_{0<t_1<t_2<\infty} 
 \frac{ \big|  \langle T[K]_{t_1}^{t_2}(f_1,f_2),f_3 \rangle \big| }{ \| f_1\|_{ L^{r_1}}  \| f_2\|_{ L^{r_2}}  \| f_3\|_{ L^{ \alpha'}} }  \,  ,
\end{equation}
then \eqref{1} is equivalent to $C_T(r_1,r_2,\alpha)<\infty$.

{
\begin{remark}\label{Rnew} If a bilinear operator of the form \eqref{BOCT} is bounded from 
$L^{r_1}\times L^{r_2} \to L^\alpha$ with $\alpha \ge 1$, then so do all of its smooth truncations with kernels
$$
K(u,v) G (u/2^{t}) G(v/2^{t'}) 
$$
uniformly on $t,t'$. Here $G$ is any   function whose Fourier transform is integrable. 
\end{remark}

To see this, we express \eqref{BOCT} in multiplier form   as follows 
$$
\int_{\mathbb R^{2n} }  \widehat{G}(\xi_1',\xi_2') \bigg[ \int_{\mathbb R^{2n} }  
 \widehat{K}(\xi_1-\xi_1',\xi_2-\xi_2')  \widehat{f_1}(\xi_1)  \widehat{f_2}  (\xi_2) e^{2\pi i x\cdot (\xi_1+\xi_2)} d\xi_1d\xi_2 \bigg]    d\xi_1'\, d\xi_2'
$$
and then we pass the $L^\alpha(dx)$ norm   on the square bracket. 
}

\begin{definition}[{\bf {Stopping collection}}
\cite{Conde-Alonso2017}] 
Let $\mathcal{D}$ be a fixed dyadic lattice in $\mathbb{R}^n$
and $Q\in \mathcal D$ be a fixed dyadic cube in ${\mathbb{R}^n}$. A collection $\mathcal Q \subset \mathcal D$ of dyadic cubes is a $stopping$ $collection$ with top $ Q$ if the elements of $\mathcal Q$ 
satisfy 
$$
L,L' \in \mathcal Q, L\cap L'\neq \emptyset\Rightarrow L=L' 
$$
$$
 L\in\mathcal Q\Rightarrow L\subset 3Q, 
$$
and   enjoy the  separation properties
\begin{enumerate}
\item[(i)] if $L,L' \in \mathcal Q$, $|s_L-s_{L'}|\geq 8$, then $7L\cap7L'=\emptyset$.

\item[(ii)] $\bigcup\limits_{\substack{ L\in\mathcal Q \\ 3L\cap 2Q\neq \emptyset}}9L\subset\bigcup\limits_{L\in\mathcal Q}L=:sh\mathcal Q$.
\end{enumerate}
Here $s_L= \log_2 \ell (L)$, where  $\ell (L)$ is the length of the cube $L$. 
\end{definition}

Let $\cic{1}_A$ be the characteristic function of a set $A$. 
	We use $M_p $ to denote the power version of the Hardy-Littlewood maximal function
	$$
	M_p(f)(x)=\sup\limits_{x\in Q}\bigg(\frac{1}{|Q|} \int_{Q} |f(y)|^p dy \bigg)^{\frac{1}{p}},
	$$
where the supremum is taken over cubes $Q\subset {\mathbb{R}^n}$ containing $x$. 

We need the following definition.
\begin{definition} [\textbf{$\mathcal Y _p(\mathcal Q)$ norm},    \cite{Conde-Alonso2017}]
Let $1\leq p\leq \infty$ and let $\mathcal Y _p(\mathcal Q)$  be the subspace of $L^p(\mathbb{R}^n)$ of functions satisfying supp $ h\subset 3Q$ and
\begin{align}  \label{dfy}
\infty>\|h\|_{\mathcal Y _p(\mathcal Q)}:=
\left\{
\begin{array}{ll}
\max\big\{\|h\cic{1}_{\mathbb{R}^n\setminus sh\mathcal Q}\|_{\infty},\sup\limits_{L\in \mathcal Q}\inf\limits_{x\in \widehat{L} }{M_p h(x)}\big\}, &p<\infty, \\
\|h\|_{\infty}, &p=\infty,  \\
\end{array}
\right.
\end{align}
where $\hat{L}$ is the (nondyadic) $2^5$-fold dilation of $L$. We also denote by $\mathcal X _{p}({\mathcal Q})$ the subspace of $\mathcal Y _{p}({\mathcal Q})$ of functions satisfying
$$b=\sum\limits_{L\in\mathcal Q}b_L, \quad\text{supp}~ b_L \subset L.$$
Furthermore, we say $b \in \dot {\mathcal X}_p(\mathcal Q)$ if 
$$b\in {\mathcal X}_p(\mathcal Q),\quad \int_Lb_L=0, \quad \forall L\in \mathcal Q. $$
 $\|b\|_{{\mathcal X}_p(\mathcal Q)}$ denotes $\|b\|_{{\mathcal Y}_p(\mathcal Q)}$ when $b\in {\mathcal X}_p(\mathcal Q)$ and similar notation for  $b \in \dot {\mathcal X}_p(\mathcal Q)$. We may omit $ \mathcal Q $ and simply write $\|\cdot\|_{{\mathcal X}_p}$ or $\|\cdot\|_{\mathcal Y_p}$.
\end{definition}

Let $a\wedge b$ denote the minimum of two real numbers $a$ and $b$. Given a stopping collection  $\mathcal {Q}$   with top cube $Q$,  we define 
\begin{align}\label{3}
\mathcal{Q}_{t_1}^{t_2}(f_1,f_2,f_3) =\frac{1}{|Q|}\Big[\langle T[K]_{t_1}^{t_2\wedge{s_Q}}(f_1\cic{1}_{Q},f_2),f_3 \rangle -\sum_{\substack{ L\in\mathcal{Q}\\ L\subset Q}}\langle T[K]_{t_1}^{t_2\wedge{s_L}}(f_1\cic{1}_{Q},f_2),f_3 \rangle \Big].
\end{align}
Then the support condition   
$$
\textup{ supp} K_s \subset \big\{(u,v)\in \mathbb{R}^{2n} :\,\, 2^{s-2}<|u|<2^s,\, 2^{s-2}<|x_2|<2^s   \big\}.
$$
 gives that 
 $$
 \mathcal{Q}_{t_1}^{t_2}(f_1,f_2,f_3)=\mathcal{Q}_{t_1}^{t_2}(f_1\cic{1}_{Q},f_2I_{3Q},f_3\cic{1}_{3Q}).
 $$
  For simplicity, we will often suppress  
  the dependence of $\mathcal{Q}_{t_1}^{t_2}$ on $t_1$ and $t_2$ by writing  
 $\mathcal{Q}(f_1,f_2,f_3)=\mathcal{Q}_{t_1}^{t_2}(f_1,f_2,f_3)$, when there is no confusion. 

\begin{lemma} [\cite{Barron2017}]\label{thm3}
Let T be a bilinear operator with kernel $K$ as the above, such that $K$ can be decomposed as in \eqref{K1} and 
suppose that the constant $C_T$ defined in \eqref{CT33} satisfies 
$$
C_T= C_T(r_1, r_2,\alpha)   <\infty 
$$
for some $1\leq r_1,r_2,\alpha<\infty$ with $1/r_1+1/r_2=1/\alpha$. 
Assume that there exist indices $1\leq p_1, p_2,p_{3}\leq \infty$ and a positive constant $C_L$ such that for all finite truncations, all dyadic lattices $\mathcal{D}$, and all stopping   collections $\mathcal{P}$ 
with top cube $Q$, the quantity   $\Lambda_{\mathcal P} (f_1,f_2,f_3)= \mathcal Q_\mu^\nu(f_1,f_2,f_3) \, |Q|$
satisfies uniformly for all $\mu<\nu$:
\begin{align}\label{lem1}
\Lambda_{\mathcal{P}}(b,g_2,g_3 )&\leq C_L|Q|\|b\|_{\dot {\mathcal X}_{p_1}}\|g_2\|_{ {\mathcal Y}_{p_2}}\|g_3\|_{ {\mathcal Y}_{p_3}}; \nonumber \\ 
\Lambda_{\mathcal{P}}(g_1,b,g_3  )&\leq C_L|Q|\|g_1\|_{{\mathcal Y}_{\infty}}\|b\|_{\dot {\mathcal X}_{p_2}}\|g_3\|_{{\mathcal Y}_{p_3}} ;  \\ 
\Lambda_{\mathcal{P}}(g_1,g_2,b )&\leq C_L|Q|\|g_1\|_{ {\mathcal Y}_{\infty}}\|g_2\|_{ {\mathcal Y}_{\infty}}\|b\|_{\dot {\mathcal X}_{p_3}} . \nonumber 
\end{align}
Then there is a constant $c_n$ depending only on the dimension $n$ 
such that the quantity   $\Lambda_\mu^\nu (f_1,f_2,f_3)= \langle T[K]_\mu^\nu  (f_1,f_2)  , f_3\rangle $ satisfies 
$$
\sup\limits_{0<\mu<\nu<\infty}|\Lambda_{\mu}^{\nu}(f_1, f_2,f_{3})|\leq c_n[C_T+C_L]\sup\limits_{{{\mathcal S}}}\mathsf{PSF}_{\mathcal S}^{\vec{p}} (f_1, f_2,f_{3})
$$
for all $f_j\in L^{p_j}(\mathbb{R}^n)$ with compact support, where $\vec{p}=(p_1,p_2,p_{3})$ and the supremum on the right is taken with respect to all sparse collections ${{\mathcal S}}$. 
\end{lemma}

Lemma \ref{thm3} is a crucial ingredient of our proof as it implies that
$$
|\langle T_{\Omega}(f_1,f_2),f_3 \rangle|\leq (C_T+C_L)\|\Omega\|_{L^{q}(\mathbb S^{2n-1})}\sup _{{\mathcal S}} \mathsf{PSF}_{{\mathcal S}}^{\vec{p}}(f_1,f_2,f_3),
$$
where $\vec{p}=(p_1,p_2,p_3)$.

Next we will consider the interpolation involving $\mathcal {Y}_q$-spaces. We only give the particular cases which we need to prove Theorem \ref{thm1}, however, more general results are available.

\begin{lemma}\label{4}
Let $0<A_2\le A_1<\infty$, $0<\epsilon<1$, and $q=1+2\epsilon$. Suppose that $\mathcal{Q}$ is a (sub)-trilinear form such that
\begin{align}
&|\mathcal{Q}(b,f,g)|\lesssim A_1 \|b\|_{\dot{\mathcal X}_1}\|f\|_{\mathcal Y_1}\|g\|_{\mathcal Y_1},\label{2.8}\\
&|\mathcal{Q}(b,f,g)|\lesssim A_2 \|b\|_{\dot{\mathcal X}_3}\|f\|_{\mathcal Y_3}\|g\|_{\mathcal Y_{3}}.\label{2.9}
\end{align}
Then we have 
$$
|\mathcal{Q}(b,f,g)|\lesssim A_1^{1-\epsilon}A_2^{{\epsilon}} \|b\|_{\dot{\mathcal X}_q}\|f\|_{\mathcal Y_q}\|g\|_{\mathcal Y_q}.
$$
\end{lemma}
\begin{proof}
Without loss of generality, we may assume $A_2 \le A_1=1$, and $\|b\|_{\dot{\mathcal X}_q}=\|f\|_{\mathcal Y_q}=\|g\|_{\mathcal Y_q}=1$, then it is enough to prove $\mathcal{Q}(b,f,g)\lesssim A_2^{{\epsilon}}$.

Fix $\lambda\geq 1$ and denote $f_{>{\lambda}}=f\cic{1}_{|f|>\lambda}$. We decompose $b=h_1+{ \ell} _1$, where
$$h_1:=\sum_{R\in \mathcal{P}}\big((b)_{>\lambda}-\frac{1}{|R|}\int_{\mathbb{R}}(b)_{>\lambda}\big)\cic{1}_{R}.$$
For $f$ and $g$, we decompose $f=h_2+{ \ell} _2$, $g=h_3+{ \ell} _3$, where $h_i:=(f_i)_{>\lambda}$, $i=2,3$.
Then it holds that
\begin {align*}
&\|h_1\|_{\dot{\mathcal X_1}}\lesssim \lambda ^{1-q}, \quad \|\ell_1\|_{\dot{\mathcal X_1}}\leq \|\ell_1\|_{\dot{\mathcal X_3}}\lesssim \lambda ^{1-\frac{q}{3}},\\
&\|h_2\|_{\mathcal Y_1}\lesssim \lambda ^{1-q}, \quad \|\ell_2\|_{\mathcal Y_1}\leq \|\ell_2\|_{\mathcal Y_3}\lesssim \lambda ^{1-\frac{q}{3}},\\
&\|h_3\|_{\mathcal Y_1}\lesssim \lambda ^{1-q}, \quad \|\ell_3\|_{\mathcal Y_1}\leq\|\ell_3\|_{\mathcal Y_{3}}\lesssim \lambda^{1-\frac{q}{3}}.
\end {align*}
The computational procedure will be put at the end of the this lemma. Now we   estimate $|\mathcal{Q}(b,f,g)|$ by the sum of the following eight terms
\begin{align*}
 \aligned {}& |\mathcal{Q}(h_1,h_2,h_3)|+|\mathcal{Q}(\ell_1,h_2,h_3)|+|\mathcal{Q}(h_1,\ell_2,h_3)|+|\mathcal{Q}(h_1, h_2, { \ell} _3)|\\&
\quad+|\mathcal{Q}(\ell_1,\ell_2,h_3)|+|\mathcal{Q}(\ell_1,h_2,\ell_3)|+|\mathcal{Q}(h_1,\ell_2,\ell_3)|+|\mathcal{Q}(\ell_1,\ell_2,\ell_3)|.\endaligned
\end{align*}
For the last term we use   assumption \eqref{2.9} while  we use     \eqref{2.8} to estimate the remaining seven terms. It follows that
$$
|\mathcal{Q}(b,f,g)|\lesssim \lambda ^{3-3q}+3\lambda ^{2-2q}+3\lambda^{1-q}+A_2\lambda^{3-q}.
$$
Noting that $ 1-q=-2\epsilon $ and $\lambda \geq1$, then we have
 \begin{align}
    |\mathcal{Q}(b,f,g)|&\lesssim 3\lambda ^{-2\epsilon}+3\lambda ^{-4\epsilon}+\lambda^{-6\epsilon}+A_2\lambda^{3-q}\nonumber\\
                   &\lesssim 7\lambda ^{-2\epsilon}+A_2\lambda^{2-2\epsilon}\nonumber\\
&\lesssim\lambda^{-2\epsilon}(7+A_2\lambda^{2})\label{2}.
  \end{align}
Let $\lambda=A_2^{-\frac{ 1}{2}}$, then $|\mathcal{Q}(b,f,g)|\lesssim A_2^{\epsilon}.$

It remains to show the estimates for $h_i$ and ${\ell} _i.$
We   only demonstrate how to compute $\|{ \ell}  _1\|_{\mathcal Y_2}\lesssim \lambda ^{1-\frac{q}{3}}$ as the estimates for $h_1,h_2,h_3,\ell _2,\ell _3$ follow in a similar way. Rewrite
  $$
 {\ell} _1=b\cic{1}_{\mathbb{R}^n \backslash sh \mathcal{P}} +\sum\limits_{R}(b)_{\leq\lambda}\cic{1}_R+\sum\limits_{R}\frac{1}{|R|}\int_{R}(b)_{>\lambda}\cic{1}_R:=I+II+III.
  $$ 
 
 From the definition in  \eqref{dfy}  we know 
 $$
 \|b\cic{1}_{\mathbb{R}^n \backslash sh \mathcal{P}}\|_{\mathcal{Y}_3}=0\lesssim \lambda ^{1-\frac{q}{3}} .
 $$
Moreover, it is easy to see that
 $$
 II=b\cic{1}_{b_{\leq \lambda}\cap sh\mathcal{P}}=b\cic{1}_S ,
 $$
  where
  $$
  S=b_{\leq \lambda}\cap sh\mathcal{P} .
  $$

 Combining \eqref{dfy} and using the H\"{o}lder's inequality, we have
 $$ \|b\cic{1}_S\|_{\mathcal{Y}_3}=\sup\limits_{R}\inf\limits_{x\in \widehat{R} }M_2b\cic{1}_S=\sup\limits_{R}\inf\limits_{x\in \widehat{R} }\sup\limits_{x\in Q}\Big(\frac{1}{|Q|}\int_{S\cap Q} |b|^3\Big)^{\frac{1}{3}}
 \leq \lambda ^{1-\frac{q}{3}} \|b\|_{\dot{\mathcal X}_q} \leq \lambda ^{1-\frac{q}{3}}.$$
Now we are in the position to consider $III$. It is easy to see that 
$$III \leq \sum\limits_{R}\frac{1}{|\widehat{R}|}\int_{\widehat{R}}(b)_{>\lambda}\cic{1}_R \leq \sum\limits_{R} \inf\limits_{x\in \widehat{R}}M_q b \cic{1} _R \leq  \sum\limits_{R}  \cic{1}_R.$$ Therefore, by the fact
$$ \|\sum\limits_{R}  \cic{1}_R\|_{\mathcal{Y}_3}\leq 1 \leq \lambda ^{1-\frac{q}{3}} , 
$$
it follows that 
$$
\| { \ell} _1\|_{\mathcal Y_3}\lesssim \lambda ^{1-\frac{q}{3}} .
$$
This finishes the proof of Lemma \ref{4}.
\end{proof} 

\section{Analysis of the kernel}\label{sec3}
In Section \ref{sect2}, we    discussed the generalized kernel $K$. Here we specialize to rough kernels. For fixed $\Omega$ in $L^r(\mathbb S^{2n-1})$ we consider the kernel
\begin{equation}\label{defKKO}
K( u,v)=\frac{\Omega\big(( u, v)/|(u,v)|\big)}{|( u, v)|^{ 2n}}. 
\end{equation}
{
We introduce the relevant notation. Define $\|[K]\|_p$ and $w_{j,p}[K]$ as follows: 
$$
\|[K]\|_p:=\sup\limits_{s\in{\mathbb{Z}}}2^{\frac{2sn}{p'}}\big(\|K_s(u,v)\|_{L^p(\mathbb{R}^{2n})}\big),
$$
$$w_{j,p}[K]=\sup\limits_{s\in \mathbb Z} 2^{\frac{2sn}{p'}}\sup\limits_{h\in {\mathbb{R}^n},|h|<2^{s-j-c_m}}(\|K_s(u,v)-K_s(u+h,v+h)\|_{p}).$$
}

 From the work in \cite{Barron2017}, we know that if the kernel  satisfies  $\|[K]\|_{p}<\infty$ and $\sum_{j=1}^{\infty}w_{j,p}[K]<\infty$, then the assumption \eqref{lem1} of Lemma \ref{thm3} holds.
However, it is difficult to verify $\|[K]\|_{p}<\infty$ and $\sum_{j=1}^{\infty}w_{j,p}[K]<\infty$ in the case $K(u,v)={\Omega((u,v)/|(u,v)|)}{|(u,v)|^{-2n}}$ with $\Omega \in L^{r}(\mathbb{S}^{2n-1})$ for $r\neq \infty$. We overcome this difficulty by using the method of Littlewood-Paley decomposition. That is, we decompose $K=\sum\limits_{j=-\infty}^{\infty} K_j$ and then actually show that each $ K_j$ satisfies the above properties.  We   establish below a key lemma 
concerning the rough kernel 
$K(u,v)={\Omega((u,v)/|(u,v)|)}{|(u,v)|^{-2n}}$.

\medskip

A bilinear Calder\'{o}n-Zygmund kernel $L$   (see \cite{{GraTor}}) 
is a function defined away from the diagonal  
on  $\mathbb{R}^{2 n}$ that satisfies 
(for some bound $A>0$) 
\begin{enumerate}
	\item the size condition 
	$$
	|L(u,v )|\leq\frac{A}{\big|(u,v)\big|^{2n}}, \qquad (u,v)\neq 0
	$$
	\item  the smoothness condition
	$$
	|L\big( (u,v)-(u',v')\big) -L(u,v )|\leq\frac{A|(u',v')|^{\epsilon}}{\big|(u,v)\big|^{2n+\epsilon}},
	$$
\end{enumerate}
when $0<\frac32 |(u',v')|\leq  |(u,v)| $, $0<\epsilon <1$. Such kernels 
give rise to bilinear Calder\'{o}n-Zygmund operators that commute with translations 
in the following way: 
$$
S(f,g)(x)=\textup{p.v.} \int_{\mathbb R^n}  \int_{\mathbb R^n}  L(x-x_1,x-x_2) f(x_1) g(x_2) \, dx_1\, dx_2.
$$

Unfortunately, if $\Omega $ lies in $L^r(\mathbb S^{2n-1})$ with $r<\infty$, then the associated $K$ given by 
\eqref{defKKO} is not a   bilinear Calder\'{o}n-Zygmund kernel, but we can decompose it as a sum 
of Calder\'{o}n-Zygmund kernels.
Given   a rough bilinear   kernel $K(u,v)={\Omega((u,v)/|(u,v)|)}{|(u,v)|^{-2n}}$  as in \eqref{defKKO},
we decompose it as follows. We fix a smooth function $\alpha$ in ${\mathbb	{R}^+}$ such that
$ \alpha(t)= 1$, for $t\in (0,1]$,
$ \alpha(t)\in (0,1)$, for $t\in (1,2)$ and
$ \alpha(t)= 0$,  for $t\in [2,\infty)$.
For $(u,v)\in {\mathbb{R}^{2n}}$ and $j\in{\mathbb{Z}}$ we introduce the functions 
$$
\beta  (u,v)=\alpha \big ( |(u,v)|\big)-\alpha\big(2 |(u,v)|\big).
$$
$$
\beta _j(u,v)=\beta \big (2^{-j} (u,v) \big) .
$$
We denote $\Delta_j$ the Littlewood-Paley operator $\Delta_j f=\mathcal{F}^{-1}(\beta_j \widehat{f})$. Here and throughout this paper $\mathcal{F}^{-1}$
denotes the inverse Fourier transform, which is defined via $$\mathcal{F}^{-1}(g)(x)=\int _{\mathbb{R}^n} g(\xi)e^{2\pi ix\cdot \xi}d\xi=\widehat{g}(-x),$$ where $\widehat{g}$ is the  Fourier transform of $g$.
Denote 
\begin{equation}\label{91}
K^i=\beta_i K
\end{equation}
and  
\begin{equation}\label{92}
K^i_j=\Delta_{j-i}K^i 
\end{equation}
for $i,j\in \mathbb{Z}$. Then we decompose the kernel $K$ as follows:
\begin{equation}\label{93}
K=\sum\limits_{j=-\infty}^{\infty}K_j,\ \quad \hbox{with\ } K_j=\sum\limits_{i=-\infty}^{\infty}K^i_j.
\end{equation}
The following lemma plays a crucial role in our analysis.

\begin{lemma}\label{lk}
	Let $K(u,v)=\Omega((u,v)/|(u,v)|){|(u,v)|^{-2n}}$ and $\Omega\in L^q(\mathbb{S}^{2n-1})$, $1<q\le \infty$,  $j\in \mathbb{Z}$. Then for any $0<\epsilon<1$, there is a constant $C_{n,\epsilon}$ such that the function 
	$$
(u,v ) \mapsto	K_j(u ,v)=\sum\limits_{i\in\mathbb{Z}}K_j^i(u ,v)
	$$
	is a bilinear Calder\'{o}n-Zygmund kernel with bound $A\leq C_{n,\epsilon}\|\Omega\|_{L^q}2^{\max(0,j)(\epsilon+2n/q)}$.
\end{lemma}

\begin{proof}
	We need to show 
	\begin{align}
&	|K_j(u,v)|\leq  C_{n,\epsilon}\|\Omega\|_{L^q} \frac{2^{\max(0,j)(\epsilon+2n/q)}}{|(u,v)|^{2n}}, \label{size} \\
&	|K_j\big((u,v)-(u',v')\big)-K_j(u,v)|\leq C_{n,\epsilon}\|\Omega\|_{L^q} \frac{2^{\max(0,j)(\epsilon+2n/q)}|(u',v')|^{\epsilon}}{|(u,v)|^{2n+\epsilon}}, \label{smooth}
	\end{align}
when $0<\frac{3}{2} |(u',v')|\leq  |(u,v)| $.
	
		Given  $x,y\in \mathbb{R}^{2n}$ with $|x|\geq \frac32 |y|>0$,  we claim that inequality \eqref{smooth} follows from
		\begin{align} 
 	|K_j^i(x-y)-K_j^i(x)|\leq C_{n,\epsilon}\|\Omega\|_{L^q}\min\Big(1,\frac{|y|}{2^{i-j}}\Big)\frac{2^{\max(0,j)2n/q}}{2^{-i\epsilon}2^{\min(j,0)\epsilon}|x|^{2n+\epsilon}}   \label{P3}
		\end{align}
for some   $\epsilon\in(0,1)$ and all $i,j\in \mathbb Z$. 

To show this claim,  let's assume for the time being that inequality \eqref{P3} is true.
	Pick  an integer $N^*$ such that $ (\log_2|y| )+j \leq N^* <( \log_2|y| )+j +1$.
	We need to consider two cases $j\ge 0$ and $j<0.$
	
	{\bf The Case for $j\ge 0$.}
	If $j\geq 0$, then for $i$ satisfies $2^{i-j}\leq |y|$, which means $i\leq N^*$. Therefore, we have 
	\begin{align*}
	\sum_{i\leq N^*}	|K_j^i(x-y)-K_j^i(x)|&\leq  \|\Omega\|_{L^q(\mathbb{S}^{2n-1})} 	\sum_{i\leq N^*}\frac{2^{j2n/q}}{2^{-i\epsilon}|x|^{2n+\epsilon}}
\\&	\leq \|\Omega\|_{L^q(\mathbb{S}^{2n-1})} \frac{2^{j(\epsilon+2n/q)}|y|^{\epsilon}}{|x|^{2n+\epsilon}}.
	\end{align*}
	
	If $j\geq 0$, then for $i$ satisfies $2^{i-j}> |y|$, which implies that $i> N^*$, it holds that
	\begin{align*}
	\sum_{i> N^*}	|K_j^i(x-y)-K_j^i(x)|&\leq  \|\Omega\|_{L^q(\mathbb{S}^{2n-1})}  	\sum_{i > N^*} \frac{|y|}{2^{i-j}}  \frac{2^{j2n/q}}{2^{-i\epsilon}|x|^{2n+\epsilon}}
\\&	\leq \|\Omega\|_{L^q(\mathbb{S}^{2n-1})} 	\frac{2^{j(\epsilon+2n/q)}|y|^{\epsilon}}{|x|^{2n+\epsilon}}.
	\end{align*}
	
	{\bf The case for $j<0$.} If $j< 0$, then for $i\leq N^*$, it holds that
	\begin{align*}
	\sum_{i\leq N^*}	|K_j^i(x-y)-K_j^i(x)|&\leq  \|\Omega\|_{L^q(\mathbb{S}^{2n-1})} 	\sum_{i\leq N^*}\frac{1}{2^{-i\epsilon}2^{j\epsilon}|x|^{2n+\epsilon}}
\\&	\leq \|\Omega\|_{L^q(\mathbb{S}^{2n-1})} \frac{|y|^{\epsilon}}{|x|^{2n+\epsilon}}.
	\end{align*}	
	
	If $j< 0$, then for $i> N^*$, we obtain
	\begin{align*}
	\sum_{i > N^*}	|K_j^i(x-y)-K_j^i(x)|&\leq  \|\Omega\|_{L^q(\mathbb{S}^{2n-1})}  	\sum_{i > N^*} \frac{|y|}{2^{i-j}}  \frac{1}{2^{-i\epsilon}2^{j\epsilon}|x|^{2n+\epsilon}}
\\&	\leq \|\Omega\|_{L^q(\mathbb{S}^{2n-1})} 	\frac{|y|^{\epsilon}}{|x|^{2n+\epsilon}}.
	\end{align*}
Summing up in all, it yields that
	\begin{align*} 
	|K_j(x-y)-K_j(x)|\leq C_{n,\epsilon}\|\Omega\|_{L^q}\frac{2^{\max(0,j)(\epsilon+2n/q)}|y|^{\epsilon}}{|x|^{2n+\epsilon}}. 
	\end{align*}
This finishes the proof of the claim. 

Therefore, to prove inequality \eqref{smooth}, it is sufficient to prove \eqref{P3}.

For $i\in \mathbb{Z}$, and $x\in \mathbb R^{2n}$,  it is easy to see that
	\begin{align*}
	|K^i(x )|&\leq \frac{\Omega\big(  x  /|x| \big) }{|x|^{2n}}\cic{1} _{\frac{1}{2}\leq \frac{|x|}{2^i}\leq 2}(x).
	\end{align*}
	Therefore
	$$
	\|K^i\|_{L^q(\mathbb R^{2n})}\leq \frac{1}{2^{2in}}\Big(\int_{2^{i-1}}^{2^{i+1}}\int_{\mathbb{S}^{2n-1}}|\Omega(\theta)|^qr^{2n-1}d\theta dr\Big)^{\frac{1}{q}}  { \, \approx\,} 2^{-2in/{q'}}\|\Omega\|_{L^q(\mathbb{S}^{2n-1})}.
	$$
Let $\Psi(x)={(1+|x|)^{-2n-1}}$ be defined on $\mathbb R^{2n}$. 	Note that
$$
| \mathcal{F}^{-1}(\beta_{i-j})(x)|\leq C_{\beta}2^{2(i-j)n}(1+2^{i-j}|x|)^{-2n-1}=C_{\beta}\Psi_{i-j}(x) , 
$$
then, using H\"{o}lder's inequality, it yields that $K_j^i= K^i * \mathcal{F}^{-1}(\beta_{i-j})$ enjoys the following property
\begin{align}
|K^i_j(x-ty)|
{  \, \lesssim\,} \|K^i\|_{L^q}\Big( \int_{2^{i-1}\leq |{z}|\leq{2^{i+1}}}|\Psi_{i-j}(x-ty-z)|^{q'}dz\Big)^{\frac{1}{q'}}, \label{P1}
\end{align}
for $x,y\in \mathbb R^{2n}$ and $t\in[0,1]$. 

Let $z=2^i z'$, for $x,y\in \mathbb R^{2n}$, it follows that
\begin{align*}
&\Big(\int_{2^{i-1}\leq |{z}|\leq{2^{i+1}}}\Big(\frac{2^{-2(i-j)n}}{(1+2^{-(i-j)}|x-ty-z|)^{2n+1}}\Big)^{q'}dz\Big)^\frac{1}{q'} \nonumber\\
&{ \, \lesssim\,} \Big(\int_{\frac{1}{2}\leq |{z'}|\leq{2}}\frac{1}{(1+2^j|\frac{x-ty}{2^i}-z'|)^{(2n+1)q'}}dz'\Big)^{\frac{1}{q'}} 2^{-2(i-j)n}2^{\frac{2in}{q'}}\\
&:=N_i^j(x,y,t).\label{P}
\end{align*}	
 If $j\leq 0$, then 
$$
N_i^j(x,y,t) { \, \lesssim\,}  \frac{1}{\big(1+2^j\max\{|\frac{x-ty}{2^i}|,1\}\big)^{2n+\epsilon}} 2^{-2(i-j)n}2^{\frac{2in}{q'}} {\, \lesssim\,} \frac{2^{2in/{q'}}2^{i\epsilon}}{2^{j\epsilon}{|x|}^{2n+\epsilon}}. 
$$
If $j>0$, we claim that 
$$
N_i^j(x,y,t){ \, \lesssim\,}  \frac{2^{2jn/q}2^{2in/{q'}}2^{i\epsilon}}{{|x|}^{2n+\epsilon}} .
$$
  Indeed, for $\frac{1}{4}\leq|\frac{x-ty}{2^i}| \leq4$, it holds that
$$
N_i^j(x,y,t){ \, \lesssim\,} 2^{-\frac{2in}{q}}2^{\frac{2jn}{q}} \leq \frac{2^{-\frac{2in}{q}}2^{\frac{2jn}{q}}}{\big(1+|\frac{x-ty}{2^i}|\big)^{2n+\epsilon}}{ \, \lesssim\,} \frac{2^{2jn/q}2^{2in/{q'}}2^{i\epsilon}}{{|x|}^{2n+\epsilon}}.
$$
As for the case $|\frac{x-ty}{2^i}|>4$ or $|\frac{x-ty}{2^i}|<\frac{1}{4}$, it follows that
$$
N_i^j(x,y,t){ \, \lesssim\,} \frac{1}{\big(1+2^j\max\{|\frac{x-ty}{2^i}|,1\}\big)^{2n+\epsilon}} 2^{-2(i-j)n}2^{\frac{2in}{q'}}{\, \lesssim\,} \frac{2^{2in/{q'}}2^{i\epsilon}}{{|x|}^{2n+\epsilon}}.
$$
	
Combining the above estimates, we deduce that
$$
|K^i_j(x-ty)|
{ \, \lesssim\,} \|\Omega\|_{L^q(\mathbb{S}^{2n-1})} \frac{2^{\max(0,j)2n/q}}{2^{-i\epsilon}2^{\min(j,0)\epsilon}|x|^{2n+\epsilon}}.
$$
This inequality further implies that
\begin{equation}
|K_j^i(x-y)-K_j^i(x)|\leq C_{n,\epsilon}\|\Omega\|_{L^q}\frac{2^{\max(0,j)2n/q}}{2^{-i\epsilon}2^{\min(j,0)\epsilon}|x|^{2n+\epsilon}}   \label{P100}
\end{equation}

On the other hand
	\begin{align*}
\big	|K_j^i(x-y)-K_j^i(x)|&=\bigg|\int_{\mathbb{R}^{2n}}K^i(z)\int_{0}^{1}2^{-2(i-j)n}(\nabla{\mathcal{F}^{-1}}\beta)(\frac{x-ty-z}{2^{i-j}})\frac{y}{2^{i-j}}\,dt\,dz\bigg|\\
&	\leq \frac{|y|}{2^{i-j}}\int_{0}^{1}\int_{\mathbb{R}^{2n}}\big|K^i(z) \big|\frac{2^{2(j-i)n}}{(1+2^{j-i}|x-ty-z|)^{2n+1}}\,dt\,dz\\
&\leq  \frac{|y|}{2^{i-j}}\int_{0}^{1}(|K^i |  *\Psi_{i-j} )(x-ty) \,dt\\
&\leq     \frac{|y|}{2^{i-j}}  \|\Omega\|_{L^q(\mathbb{S}^{2n-1})} \frac{2^{\max(0,j)2n/q}}{2^{-i\epsilon}2^{\min(j,0)\epsilon}|x|^{2n+\epsilon}}.
		\end{align*}

This estimate, together with inequality \ref{P100}, yields the inequality \ref{P3}
and hence inequality \ref{smooth} holds.
	  
	For the size condition \eqref{size}, we may let $t=0$ in \eqref{P1}. Thus
		\begin{align*}
	\sum\limits_{i\in\mathbb{Z}}	|K^i_j(x)|
&	\leq \|\Omega\|_{L^q(\mathbb{S}^{2n-1})}	\sum\limits_{i\in\mathbb{Z}} \Big(\int_{\frac{1}{2}\leq |{z'}|\leq{2}}\frac{1}{(1+2^j|\frac{x}{2^i}-z'|)^{(2n+\epsilon)q'}}dz'\Big)^{\frac{1}{q'}} 2^{-2(i-j)n}\\
&	{ \, \lesssim\,} \|\Omega\|_{L^q(\mathbb{S}^{2n-1})}\sum\limits_{i<\widetilde{N}^{\ast}} 2^{-2(i-j)n}\Big(\int_{\frac{1}{2}\leq |{z'}|\leq{2}}\frac{1}{\big(1+2^j|\frac{x}{2^i}-z'|\big)^{(2n+\epsilon)q'}}dz'\Big)^{\frac{1}{q'}}\nonumber \\
&\quad+\|\Omega\|_{L^q }\sum\limits_{i>\widetilde{N}^{\ast}} 2^{-2(i-j)n}\nonumber\\
& { \, \lesssim\,}  \|\Omega\|_{L^q(\mathbb{S}^{2n-1}) }\frac{1}{|x|^{2n}}+\|\Omega\|_{L^q(\mathbb{S}^{2n-1}) } \frac{2^{\max(0,j)2n/q}}{2^{\min(j,0)\epsilon}|x|^{2n+\epsilon}} \sum\limits_{i<\widetilde{N}^{\ast}}2^{i\epsilon}\nonumber\\
&{ \, \lesssim\,}  \|\Omega\|_{L^q(\mathbb{S}^{2n-1}) } \frac{2^{\max(0,j)(2n/q+\epsilon)}}{|x|^{2n}},
		\end{align*}
where $\widetilde{N}^*$ is the number such that $2^{\widetilde{N}^*}\approx 2^{\min(j,j/q')}|x| $. 

Therefore,  we know that $K_j$ is a bilinear Calder\'{o}n-Zygmund kernel with bound $ C_{n,\epsilon}\|\Omega\|_{L^q}2^{\max(0,j)(\epsilon+2n/q)}$. The proof of this lemma is finished.
\end{proof}

\section{the proof of Theorem \ref{thm1}}\label{sec4}
We begin by stating a   known result. 

\begin{proposition} [\cite{Grafakos2018}]\label{p1}
	Let  $1\leq p_1,p_2<\infty$ and  $1/p=1/{p_1}+1/{p_2}$. Let $\Omega$ be in ${L^{q}(\mathbb{S}^{2n-1})}$ with  $1<q\leq \infty$ and let  $\delta \in(0,1/{q'})$. Let  $T_j$ be the bilinear  Calder\'on-Zygmund  operator with kernel $K_j$. Them, for $j\leq 0$, the operator $T_j$ maps $L^{p_1}(\mathbb{R}^n)\times L^{p_2}(\mathbb{R}{^n})$ to $L^p(\mathbb{R}{^n})$
	with norm $C\|\Omega\|_{L^{q}}2^{-|j|(1-\delta )}$.
\end{proposition}

The following lemma will be crucial in dealing with the adjoints of $T_\Omega$. The ingredients of its proof are
contained in some known works but the precise statement below may not have appeared in the literature. 

\begin{lemma} \label{Lemma-last}
Let  $1 \le  q <4$, $\delta>0$, and let $b$ be a smooth function on $\mathbb R^{2n}$ which satisfies:
\begin{enumerate}
\item[(a)] $ \| b\|_{L^{q}(\mathbb R^{2n})} \le C_*$,
\item[(b)] $ |b(\xi,\eta) | \le C_* \min(|(\xi,\eta)|, |(\xi,\eta)|^{-\delta})    $,
\item[(c)] $ |\p^{\alpha} b(\xi,\eta) | \le C_\alpha C_* \min(1, |(\xi,\eta)|^{-\delta})    $.
\end{enumerate}
Let $\beta$ be a smooth function supported in an annulus in $\mathbb R^{2n}$ and let $\beta _j(y,z)=\beta \big (2^{-j} (y,z) \big) $ for $j\in \mathbb Z$. Then the multiplier 
$$
b_j(\xi,\eta) = \sum_{i\in \mathbb Z} \beta_{j-i} (\xi,\eta) b (2^i (\xi,\eta)) 
$$
satisfies
$$
\|T_{b_j} \|_{L^2\times L^2\to L^1} \lesssim j \, C_*\, 2^{-\delta j (1-\frac q4)}   .
$$
\end{lemma}
\begin{proof}

Denote $
	b_{j,0}=\beta_{j} (\xi,\eta) b (\xi,\eta) 
	$ and write $b_j=b_j^1+b_j^2$, where $b_j^1$ is the diagonal part of $b_j$ according to the wavelet 
decomposition in \cite[Section 4]{Grafakos2019} and $b_j^2$ is the off-diagonal part.  (In this reference 
$b$ is denoted by $m$, $b_j$ by $m_j$ and $b_{j,0}$ by $m_{j,0}$.)

Let 
$$
		C_0=\max\limits_{|\alpha|\leq \lfloor \frac{2n}{4-q'} \rfloor +1} \|\partial ^{\alpha}b_{j,0}\|_{L^{\infty}} \lesssim C_{*}2^{-\delta j}.
$$
	  By  \cite[Section 4]{Grafakos2019},  we obtain 
	$$
	\|T_{b^1_{j}}\|_{L^2\times L^2\rightarrow L^1}{ \, \lesssim\,} jC_0^{1-\frac{q}{4}}\|b_{j,0}\|_{L^{q}}^{\frac{q}{4}} 
	{ \, \lesssim\,} jC_0^{1-\frac{q}{4}}\|b\|_{L^{q}}^{\frac{q}{4}}   
	{ \, \lesssim\,} j (C_{*}2^{-\delta j})^{1-\frac{q}{4}}\|b\|_{L^{q}}^{\frac{q}{4}}
	{ \, \lesssim\,} j C_{*} (2^{-\delta j})^{1-\frac{q}{4}} .
	$$

A similar estimate (without $j$) holds for the off-diagonal part $T_{b^2_{j}}$   by the same procedure  as in \cite[Section 5]{Grafakos2018}.  It follows that
$$
\|T_{b^2_j}\|_{L^2\times L^2\rightarrow L^1}{ \, \lesssim\,}  2^{-\delta j}\|b_{j,0}\|_{L^q(\mathbb{R}^{2n })} 
{ \, \lesssim\,}  C_* 2^{-\delta j}  .
$$
Combining the estimates for $b_j^1$ and $b_j^2$, we obtain
$$ 
		\|T_{b_{j}}\|_{L^2\times L^2\rightarrow L^1}
		\lesssim j C_* 2^{-\delta j( 1-\frac{q}{4}) }  .  
$$

\end{proof}

We also need the following lemma. 

\begin{lemma}\label{p2}
	Let  $2\leq p_1,p_2\leq \infty$, $1\le p\le 2$, $1/p=1/{p_1}+1/{p_2}$, $\Omega\in{L^{q}(\mathbb{S}^{2n-1})}$. For {$ 
		 j > 0$} we have that
	\begin{eqnarray*}
			\|T_j\|_{ L^{p_1}(\mathbb{R}^n)\times L^{p_2}(\mathbb{R}{^n}) \rightarrow L^p(\mathbb{R}{^n}) }\lesssim
		\begin{cases}
		Cj2^{-j\delta (1-\frac{q'}{4})}\|\Omega\|_{L^{q}(\mathbb{S}^{2n-1})} , &\frac{4}{3}< q\leq 2 , \delta < \frac{1}{q'};\cr C  j 2^{-j\delta \frac{1}{2}}\|\Omega\|_{L^q(\mathbb{S}^{2n-1})} , &q>2, \delta <1/2 .\end{cases}
	\end{eqnarray*}
\end{lemma}
\begin{proof}
	The techniques of the proof are   borrowed from \cite{Grafakos2019}. 
	Introduce the notation:   
	$$
m = \widehat{K^0}, \quad 	  m_j=\widehat{K_j}, \quad m_{j,0}=\widehat{K^0}\beta _j  ,
	$$
	where $K^0$, $\beta_j$, and $K_j$ are the same as in \eqref{91}, \eqref{92}, and \eqref{93} are associated 
	with the fixed $\Omega $ in $L^q(\mathbb S^{2n-1})$.

We first fix $q$ satisfying $   4/3< q \leq 2$.  
As $q\le 2$,  the Hausdorff-Young inequality yields that
	$$
	\|m \|_{L^{q'}} \leq \|K^0 \|_{L^q} { \, \lesssim\,} 
	\|\Omega\|_{L^q{  (\mathbb{S}^{2n-1})} }.
	$$
Also, it is not too hard to verify that conditions (b) and (c) in Lemma~\ref{Lemma-last} hold (see \cite[Lemma 6.4]{Grafakos2019}) with 
$C_* = \|\Omega \|_{L^q(\mathbb S^{2n-1})}$  and $\delta<1/q'$.  Applying  Lemma \ref{Lemma-last}   we obtain 
	$$ 
		\|T_{m_{j}}\|_{L^2\times L^2\rightarrow L^1}
		\lesssim j 2^{-\delta j( 1-\frac{q'}{4}) }\|\Omega\|_{L^q(\mathbb{S}^{2n-1})} .
	$$

Now let 
$$
 {(m_j)^{*1}}(\xi_1,\xi_2)=m_j (-(\xi_1+\xi_2),\xi_2) , \quad 
  {(m_j)^{*2}}=m_j (\xi_1,-(\xi_1+\xi_2))
 $$
  be the two adjoint  multipliers associated with   $m_j$. Then 
  we have 
  $$
  (m_j)^{*1} =\sum\limits_{i} { (\beta _{j-i}  \circ A^t } ) \, ( \widehat{\beta _i  K }\circ A^t )
  =\sum\limits_{i} { (\beta _{j-i}  \circ A^t } ) \,  \widehat{\beta   K }  ( A^t 2^i ( \cdot)  )
  $$
  where  $A= \begin{pmatrix} -I_n & -I_n \\ 0 & I_n \end{pmatrix}$, and $I_n$ is 
  the $n\times n$ identity matrix. 
  
 We now notice that the function $b(\xi,\eta) = \widehat{\beta   K }  ( A^t(\xi,\eta) ) $ satisfies the 
 hypotheses of Lemma~\ref{Lemma-last} as $A^t(\xi,\eta)$ has the same size as $(\xi,\eta)$. 
(Here $(\xi,\eta)$ is thought of as a column vector.)
 The same argument works for the other adjoint of $m_j$ with the matrix 
 $  \begin{pmatrix} I_n & 0 \\ -I_n & -I_n \end{pmatrix}$ in place of $A$.  
  It follows that 
$$
\|T_{({m_{j})}^{*1}}\|_{L^2\times L^2\rightarrow L^1} {  + }\|T_{({m_{j})}^{*2}}\|_{L^2\times L^2\rightarrow L^1}{ \, \lesssim\,}  j 2^{-j\delta(1-\frac{q'}{4})}\|\Omega\|_{L^q(\mathbb{S}^{2n-1})} .
$$ 
By duality, we have
$$
\|T_{{m_{j}}}\|_{L^{\infty}\times L^2\rightarrow L^2} {  + } \|T_{{m_{j}}}\|_{L^2\times L^{\infty}\rightarrow L^2}{ \, \lesssim\,}  j 2^{-j\delta(1-\frac{q'}{4})}\|\Omega\|_{L^q(\mathbb{S}^{2n-1})} .
$$ 
For $4/3<q\leq2$, interpolating between the above two estimates implies that  
$$
\|T_{m_j}\|_{L^{p_1}\times L^{p_2}\rightarrow L^p}{ \, \lesssim\,}  j 2^{-j\delta(1-\frac{q'}{4})}\|\Omega\|_{L^q(\mathbb{S}^{2n-1})} 
, \quad \delta <\frac{1}{q'},
$$
where $2\leq p_1,p_2\leq \infty$, $1\leq p\leq 2$ and $\frac{1}{p_1}+\frac{1}{p_2}=\frac{1}{p}$. 

Now for $q>2$, thanks to the embedding $L^{q}(\mathbb{S}^{2n-1}) \subseteq  L^{2}(\mathbb{S}^{2n-1}) $, we have
$$
\|T_{m_j}\|_{L^{p_1}\times L^{p_2}\rightarrow L^p}{ \, \lesssim\,}  j 2^{-j\delta(1-\frac{2}{4})}\|\Omega\|_{L^2(\mathbb{S}^{2n-1})}{ \, \lesssim\,}  j 2^{-j\delta \frac{1}{2}}\|\Omega\|_{L^q(\mathbb{S}^{2n-1})} ,\quad 
\delta <\frac{1}{2},
$$
where $2\leq p_1,p_2\leq \infty$, $1\leq p\leq 2$ and $\frac{1}{p_1}+\frac{1}{p_2}=\frac{1}{p}$. 

This completes the proof of this lemma. 
\end{proof}

We are now in the position to prove  Theorem \ref{thm1}. 
\begin{proof}[Proof of Theorem \ref{thm1}] 
{	
By Littlewood-Paley decomposition of the kernel,  $T_{\Omega}$ can be written as 
	$$T_{\Omega}(f,g)(x)=
\sum\limits_{j=-\infty}^{\infty}\int_{\mathbb{R}^n}\int_{\mathbb{R}^n}|K_j(x-y,x-z) f(y)g(z)\,dydz:=\sum\limits_{j=-\infty}^{\infty}T_{j}(f,g)(x).$$ 

Given a stopping collection  $\mathcal {Q}$   with top cube $Q$,  let $\mathcal Q_j$ be defined as 
\begin{align*}
\mathcal{Q}_{j,t_1}^{t_2}(f_1,f_2,f_3) =\frac{1}{|Q|}\Big[\langle T[K_j]_{t_1}^{t_2\wedge{s_Q}}(f_1\cic{1}_{Q},f_2),f_3 \rangle -\sum_{\substack{ L\in\mathcal{Q}\\ L\subset Q}}\langle T[K_j]_{t_1}^{t_2\wedge{s_L}}(f_1\cic{1}_{Q},f_2),f_3 \rangle \Big].
\end{align*}
For the sake of simplicity, let's denote
$\mathcal{Q}_j(f_1,f_2,f_3)=\mathcal{Q}_{j,t_1}^{t_2}(f_1,f_2,f_3)$.	
\medskip

   Our proof  will be divided into two parts $\sum_{j>0}T_{j}$ and $\sum_{j\leq 0}T_{j}$.  Each part  should satisfy the  assumption \eqref{lem1} of Lemma \ref{thm3}. We therefore consider these two parts into two steps.
	 }
	
 {\bf{Step 1. Estimate for $j >0$}}.
 
 	{
 			Fix $0<\gamma<1$,  by Lemma \ref{lk}, $T_j$ 	is a bilinear Calder\'{o}n-Zygmund operator with kernel $K_j$,
 			and the size and smoothness conditions constant $A_j\leq C_{n,\gamma}\|\Omega\|_{L^r}2^{j(\gamma+2n/r)}$.  
 			
 			\smallskip
 		Combining the methods in  \cite[Section 3] {Barron2017}, we know the kernel of $T_j$ satisfies
 		$\|[K_j]\|_{p}\lesssim2^{j(\epsilon+2n/q)} <\infty$ for fixed $j\in\mathbb{Z}$.
 		This enables us to use  Lemma 3.1 and Proposition 3.3 in \cite{Barron2017} with $A_j\leq C_{n,\epsilon}\|\Omega\|_{L^r}2^{ j(\gamma+2n/r)}$ (Then choose $\beta=1$ and $p=1$).  }
	Hence 	
	$$
	|\mathcal{Q}_j(f_1,f_2,f_3)| { \, \lesssim\,}  \|\Omega\|_{L^r}2^{j(\gamma+2n/r)}|Q| \|f_1\|_{\dot{\mathcal X}_1}\|f_2\|_{\mathcal Y_1}\|f_3\|_{\mathcal Y_1} .
	$$
By Lemma \ref{p2}, choosing $p_1=p_2=3$, we have 
$$|\mathcal{Q}_j(f_1,f_2,f_3)| { \, \lesssim\,}  \|\Omega\|_{L^r}j2^{-cj}|Q| \|f_1\|_{\dot{\mathcal X}_3}\|f_2\|_{\mathcal Y_3}\|f_3\|_{\mathcal Y_3},$$
where $c<{1}/{r'} (1-{r'}/{4})$, if $4/3<r \leq 2$ and  $c<1/4$ if $r>2$. 

Interpolating via Lemma \ref{4}, it follows that  for any $0<\epsilon<1$ there  exits $q=1+2\epsilon$ so that 
\begin{align*}
|\mathcal{Q}_j(f_1,f_2,f_3)|&{ \, \lesssim\,}  \|\Omega\|_{L^r} 2^{j(\gamma+2n/r)(1-\epsilon)}j^\epsilon
2^{-cj\epsilon}|Q| \|f_1\|_{\dot{\mathcal X}_q}\|f_2\|_{\mathcal Y_q}\|f_3\|_{\mathcal Y_q}\\
&{ \, \lesssim\,}  j2^{-j\gamma\epsilon} 2^{j(\gamma+2n/r)}2^{-(c+2n/r)j\epsilon} \|\Omega\|_{L^r}|Q| \|f_1\|_{\dot{\mathcal X}_q}\|f_2\|_{\mathcal Y_q}\|f_3\|_{\mathcal Y_q}.
\end{align*}
If we choose $\gamma < c$	and $\epsilon=\frac{2n/r+\gamma}{2n/r+c}$, then $0<\epsilon<1$.
Therefore
$$|\mathcal{Q}_j(f_1,f_2,f_3)|{ \, \lesssim\,} j2^{-j\gamma\epsilon} |Q|  \|\Omega\|_{L^r}\|f_1\|_{\dot{\mathcal X}_q}\|f_2\|_{\mathcal Y_q}\|f_3\|_{\mathcal Y_q}.$$
Summing over $j\in \mathbb{Z}^+$, we can conclude that for $q=1+ 2\frac{2n/r+\gamma}{2n/r+c}$
$$|\mathcal{Q}(f_1,f_2,f_3)|{ \, \lesssim\,} |Q|  \|\Omega\|_{L^r}\|f_1\|_{\dot{\mathcal X}_q}\|f_2\|_{\mathcal Y_q}\|f_3\|_{\mathcal Y_q}. $$
By symmetry, it also yields that
\begin{align*}
&|\mathcal{Q}(f_1,f_2,f_3)|{ \, \lesssim\,} |Q|  \|\Omega\|_{L^r}\|f_1\|_{{\mathcal Y}_q}\|f_2\|_{\dot{\mathcal X_q}}\|f_3\|_{\mathcal Y_q}, \\
&|\mathcal{Q}(f_1,f_2,f_3)|{ \, \lesssim\,} |Q|  \|\Omega\|_{L^r}\|f_1\|_{{\mathcal Y}_q}\|f_2\|_{\mathcal Y_q}\|f_3\|_{\dot{\mathcal X_q} }.
\end{align*}

{\bf{Step 2. Estimate for $j \le  0$}}.

By Lemma \ref{lk}, 
$T_j$ 	is a bilinear Calder\'{o}n-Zygmund kernel with constant $A_j\leq \|\Omega\|_{L^r}$. Hence
$$
|\mathcal{Q}_j(f_1,f_2,f_3)|{ \, \lesssim\,}  \|\Omega\|_{L^r}|Q| \|f_1\|_{\dot{\mathcal X}_1}\|f_2\|_{\mathcal Y_1}\|f_3\|_{\mathcal Y_1}.
$$

By Proposition \ref{p1} with $p_1=p_2=2$, we have 
$$
|\mathcal{Q}_j(f_1,f_2,f_3)|{ \, \lesssim\,}  \|\Omega\|_{L^r}2^{-c|j|}|Q| \|f_1\|_{\dot{\mathcal X}_2}\|f_2\|_{\mathcal Y_2}\|f_3\|_{\mathcal Y_{\infty}},
$$
where $c=1- \delta $, $\delta<1/{q'}$. For any $q>1$,  by Lemma 4.3 and Lemma 4.4 in \cite{Barron2017}, then summing over $j\leq 0$, one obtains
$$
|\mathcal{Q}(f_1,f_2,f_3)|{ \, \lesssim\,} |Q|  \|\Omega\|_{L^r}\|f_1\|_{\dot{\mathcal X}_q}\|f_2\|_{\mathcal Y_q}\|f_3\|_{\mathcal Y_q} .
$$
$$
|\mathcal{Q}(f_1,f_2,f_3)|{ \, \lesssim\,} |Q|  \|\Omega\|_{L^r}\|f_1\|_{{\mathcal Y}_q}\|f_2\|_{\dot{\mathcal X_q}}\|f_3\|_{\mathcal Y_q} .
$$
$$
|\mathcal{Q}(f_1,f_2,f_3)|{ \, \lesssim\,} |Q|  \|\Omega\|_{L^r}\|f_1\|_{{\mathcal Y}_q}\|f_2\|_{\mathcal Y_q}\|f_3\|_{\dot{\mathcal X_q} }.
$$
In conclusion, the above two steps hold for 
$$	p>	\begin{cases}
		\frac{24n+3r-4}{8n+3r-4}, & \frac{4}{3}< r \leq 2 ;
	\cr	\frac{24n+r}{8n+r}, & r>2 .
			\end{cases}$$
			since the norm of $\mathcal{Y}_q$ is increasing over $q$.
			
{ Using Theorem A, we can find $r_1$, $r_2$ in $ [2,\infty]$ and $\alpha$ in $[1,2]$ such 
$T_{\Omega}$ maps $L^{r_1}\times L^{r_2}$ to $L^\alpha$. But a smooth truncation of the kernel $K(u,v)$ also 
gives rise to an operator with a similar bound (see Remark~\ref{Rnew}), thus we have that $C_T(r_1,r_2,\alpha)<\infty$ and \eqref{1} is valid.}
Hence,  $T_{\Omega}$ satisfies Lemma \ref{thm3}. Moreover, we can choose  $c<\frac{1}{r'}(1-\frac{r'}{4})$ if $\frac{4}{3}< r \leq 2$, and $c<\frac{1}{4}$ if $r>2$, such that $p>3- \frac{2c}{2n/r+c}$. Then
$$
|\mathcal{Q}(f_1,f_2,f_3)|{  \, \lesssim\,} \|\Omega\|_{L^r}\sup\limits_{\mathcal S}\mathsf{PSF}_{\mathcal S,\vec{p}}(f_1, f_2,f_{3}),
$$
 this finishes the proof of Theorem \ref{thm1}, since the multiplication operators regarding  the remaining   truncations satisfy the required   $\mathsf{PSF}_{\mathcal S}^{(1,1,1)}$  bound \cite[Section 6.2]{Barron2017}. 
\end{proof}
\section{derivation of the Corollaries}\label{sec5}

\begin{proof}[Proof of Corollary \ref{collary2}]
		The techniques  are   borrowed from \cite{Culiuc2018}, but  the weight classes  are different.
		
Define $\sigma=v_{\vec{w}}^{-\frac{q'}{q}}$ and choose $p_i> \max\{\frac{24n+3r-4}{8n+3r-4}, \frac{24n+r}{8n+r}\}$, with $p_i <q_i$, $i=1,2$ and $p'_3>q$.
By Theorem \ref{thm1} and duality, for any sparse collection $\mathcal S$, it is enough to show that
\begin{equation}\label{5.1.Zhidan}
\mathsf{PSF}_{\mathcal  S}^{(p_1,p_2,p_3)}(f_1,f_2,f_3)\lesssim \prod\limits_{i=1}^2\|f_i\|_{L^{q_i}(v_i)}\|f_3\|_{L^{q'}(\sigma) }
\end{equation}
with bounds independent of $\mathcal S$.
 
Let 
$$
w_1=v_1^{\frac{p_1}{p_1-q_1}} , \quad w_2=v_2^{\frac{p_2}{p_2-q_2}}, \quad w_3=\sigma ^{\frac{p_3}{p_3-q'}}
$$
and $f_i=g_i w_i^{\frac{1}{p_i}}$, $i=1,2,3$. Then we have 
$$
\|f_i\|_{L^{q_i}(v_i)}=\|g_i\|_{L^{q_i}(w_i)} , \qquad i=1,2 ,
$$
 and 
 $$
 \|f_3\|_{L^{q'}(\sigma)}=\|g_3\|_{L^{q'}(w_3)} . 
 $$
Let $q_3=q'$. It follows that
\begin{align*}
&  \mathsf{PSF}_{\mathcal  S}^{(p_1,p_2,p_3)}(f_1,f_2,f_3)\\
&=   \mathsf{PSF}_{\mathcal  S}^{(p_1,p_2,p_3)}\big(g_1w_1^{\frac{1}{p_1}},g_2w_2^{\frac{1}{p_2}},g_3w_3^{\frac{1}{p_3}}\big)\\
&=\sum\limits_{Q\in {\mathcal  S}} \Big(\prod\limits_{j=1}^3 w_j(E_Q)^{\frac{1}{q_j}}\Big(\frac{\langle g_j^{p_j}w_j \rangle _Q}{\langle w_j \rangle _Q}\Big)^{\frac{1}{p_j}}\Big)\times \Big(\prod\limits_{j=1}^3{ \big(  \langle w_j \rangle _Q}\big)^{\frac{1}{p_j}-\frac{1}{q_j}}\Big)
\times \Big(|Q|\prod\limits_{j=1}^3\Big(\frac{{\langle w_j \rangle_Q}}{w_j(E_Q)}\Big)^{\frac{1}{q_j}}\Big).
\end{align*}
By a simple calculation, we have 
$$
\prod\limits_{j=1}^2{\langle w_j \rangle _Q}^{\frac{1}{p_j}-\frac{1}{q_j}}{\langle w_3 \rangle _Q}^{\frac{1}{p_3}-\frac{1}{q'}}=
 \prod\limits_{j=1}^2{\langle w_j \rangle_Q}^{\frac{1}{p_j}-\frac{1}{q_j}}{\langle v_{\vec{w}}^{\frac{p'_3}{p'_3-q}}\rangle _Q^{\frac{1}{q}-\frac{1}{p'_3}}}=[\vec{v}]_{A_{\vec{q},\vec{p}}}.
 $$
We now   deal with the second product   using the technique in \cite{Lerner2019}. Let 
$$
x_1=\frac{p_1-q_1}{p_1q_1},\quad x_2=\frac{p_2-q_2}{p_2q_2} , \quad x_3=\frac{p_3-q'}{p_3q'} ,
$$
 then 
 $$
 w_1^{-\frac { x_1} 2}w_2^{-\frac { x_2}2}w_3^{-\frac { x_3}2}=1 . 
 $$
    H\"{o}lder's inequality and the fact that 
$$
-\frac{x_1}{2}-\frac{x_2}{2}-\frac{x_3}{2}+\frac{1}{2p'_1}+\frac{1}{2p'_2}+\frac{1}{2p'_3}=1
$$
 imply that   
$$
\prod\limits_{i=1}^3\big(w_i(E_Q)\big)^{-\frac {x_i}2}E_Q^{\frac{1}{2p'_i}}\geqslant\int_{E_Q}\prod\limits_{i=1}^3w_i^{-\frac {x_i}2}=|E_Q|.
$$
The sparseness of $\mathcal S$ yields that
$$
\prod\limits_{i=1}^3\Big(\frac{w_i(E_Q)}{|Q|}\Big)^{-\frac {x_i}2}\geq \eta ^{-\frac{ x_1}2-\frac{x_2}2-\frac{x_3}2} . 
$$
Therefore
$$
\prod\limits_{i=1}^3\Big(\frac{{\langle w_i \rangle _Q}}{\frac{1}{|Q|}w_i(E_Q)}\Big)^{ -\frac{x_i}2 }\leq 
\eta^{\frac{x_1}2+\frac{x_2}2+\frac{x_3} 2}\prod\limits_{i=1}^3{{\langle w_i\rangle _Q}}^{-\frac{x_i}2} .
$$
By Definition\ref{def1}, we have
$$\begin{aligned}
\prod\limits_{i=1}^3\Big(\frac{{\langle w_i \rangle _Q}}{\frac{1}{|Q|}w_i(E_Q)}\Big)^{\frac{1}{q_i}}&\leq \big(\eta^{{x_1+x_2+x_3}}\prod\limits_{i=1}^3{{\langle w_i \rangle _Q}}^{-x_i}\big)^{\max (-\frac{1}{x_i q_i})}\\& \leq \big(\eta^{{x_1+x_2+x_3}}[\vec{v}]_{A_{\vec{q},\vec{p}}})^{\max (-\frac{1}{x_i q_i})}.\end{aligned}
$$
Finally note that, by \cite{Culiuc2018}, the first product depends on the $L^{q_j}(w_j)$-boundedness of $M_{p_j,w_j} $, where
$$ 
M_{p_j,w_j}f(x) = \sup\limits_{Q\ni x} \Big(\frac{1}{|w(Q)|} \int_{Q}|f|^{p_j}w_j \Big)^{\frac{1}{p_j}}.
$$
This concludes the proof of \eqref{5.1.Zhidan}
\end{proof}

\begin{proof}[Proof of Corollary \ref{collary1}]
	For $2<p<\infty$, let $\sigma=w^\frac{-2}{2-p}$, 
	 $\rho=\frac{p}{p-2}$  and choose $q_i> \max\{\frac{24n+3r-4}{8n+3r-4}, \frac{24n+r}{8n+r}\}$ such that $q_i <\rho$, and $q_i<p$.
	By Theorem \ref{thm1} and duality, it is enough to prove that 
	for any sparse collection $\mathcal S$, we have
	$$
	\mathsf{PSF}_{\mathcal  S}^{(q_1,q_2,q_3)}(f_1,f_2,f_3)\lesssim 
 \prod\limits_{i=1}^2\|f_i\|_{L^{p}(w)}\|f_3\|_{L^{\rho}(\sigma) }
	$$
	with bounds independent of $\mathcal S$. 
	The proof of this fact is omitted as it follows from the same steps as in   Section 5.1 in \cite{Barron2017}.	\end{proof}

Next, we provide another corollary which is related to Corollary 1.7 in \cite{Culiuc2018}.

\begin{corollary}\label{corollary3}
		Suppose $\Omega \in L^{r}(\mathbb{S}^{2n-1})$ with vanishing integral and $r>4/3$. For  $p_1,p_2>\max\{\frac{24n+3r-4}{8n+3r-4}, \frac{24n+r}{8n+r}\}$, $\frac{1}{p}=\frac{1}{p_1}+\frac{1}{p_2}$ with $1<p<\max\{\frac{24n+3r-4}{16n},\frac{24n+r}{16n}\}$. Then for weights $w_1^2 \in A_{p_1}$, $w_2^2 \in A_{p_2}$, $w=w_1^{\frac{p}{p_1}}w_2^{\frac{p}{p_2}}$, there exists a constant $C= C_{w,p_1,p_2,n,r} $ such that
$$
\|T_{\Omega}(f_1,f_2)\|_{L^{p}(w)}\leq C \|\Omega\|_{L^r}\|f_1\|_{L^{p_1}(w_1)}\|f_2\|_{L^{p_2}(w_2)}.
$$
\end{corollary}

We end this section with another corollary concerning the commutator of a rough $T_\Omega$ with a pair of BMO functions $\vec b = (b_1,b_2)$. For a pair $\vec \alpha = (\alpha_1, \alpha_2)$ of nonnegative integers, we define this commutator 
(acting on a pair of nice functions $f_j$) as follows: 
\[
\big[T_{\Omega}, \vec b\, \big]_{\vec \alpha}  (f_1,f_2) (x) =\textup{p.v.} \int_{\mathbb R^{2n} } \frac{ \Omega ((y_1,y_2)')}{|(y_1,y_2)|^{2n} }
f_1(x-y_1) f_2(x-y_2)\prod_{i=1}^2(b_i(x)-b_i(y_i))^{\alpha_i}    dy_1 dy_2
\]
As a consequence of Proposition 5.1 in \cite{LMO2018} and of Corollary~\ref{collary2}, 

\begin{corollary}\label{collary-last}
Let $\Omega \in L^{r}(\mathbb{S}^{2n-1})$ with $r>4/3$ and  $\int _{\mathbb{S}^{2n-1}}\Omega \, d\sigma=0$.  Let $\vec{q}=(q_1,q_2)$, $\vec{p}=(p_1,p_2,p_3)$ with $\vec{p}\prec \vec{q}$ and $p_i>\max\{\frac{24n+3r-4}{8n+3r-4}, 	\frac{24n+r}{8n+r} \}$,  $i=1,2,3$. Let 
$$
\mu_{\vec{v}}=\prod_{k=1}^{2}v_k^{q/q_k}
$$
 and $\frac{1}{q}=\frac{1}{q_1}+\frac{1}{q_2}$, $1< q <\max\{\frac{24n+3r-4}{16n}, \frac{24n+r}{16n}\}$ and 
let $q_3=q'$. Then
there is a constant $C=C_{\vec p ,\vec q, r,n, \vec \alpha}$ such that
	$$
	\big\| \big[T_{\Omega}, \vec b\, \big]_{\vec \alpha} (f,g)\big\|_{L^{q}(\mu_{\vec{v}})}\leq C \|\Omega\|_{L^r} [\vec{v}]^{\max_{1 \leq i\leq 3}\{\frac{p_i}{q_i-p_i}\}}_{A_{\vec q,\vec p}}   \|f\|_{L^{q_1}(v_1)}\|g\|_{L^{q_2}(v_2)}  \prod_{i=1}^2\| b_i\|_{BMO}^{\alpha_i} .
	$$
\end{corollary}

\end{document}